\documentclass{amsart}
\usepackage{fullpage,amsthm, amsmath, amssymb, amsfonts}
\usepackage{ragged2e}
\usepackage{fixltx2e}
	\usepackage{setspace}
\usepackage{amssymb}
\usepackage{enumerate}
	\usepackage{times}
	\usepackage[none]{hyphenat}
\usepackage{tkz-graph}
\usepackage{tikz}
\usepackage{tikz-cd}
\usepackage{booktabs}
\usepackage{subfigure}
\usepackage{indentfirst}
\usepackage{leftidx}
\usepackage[top=1in, bottom=1in, left=1.5in, right=1in]{geometry}
\usepackage{setspace}
\usepackage{graphicx}
\usepackage{mathtools}

  \newtheorem{theorem}{Theorem}
  \newtheorem{proposition}[theorem]{Proposition}
  \newtheorem{lemma}[theorem]{Lemma}
  \newtheorem{corollary}[theorem]{Corollary}
  
  \newtheorem{definition}[theorem]{Definition}
  \newtheorem{example}[theorem]{Example}
  \newtheorem{remark}[theorem]{Remark}
\newcommand{\newword}[1]{\textbf{\emph{#1}}}
\newcommand*{\defeq}{\stackrel{\text{def}}{=}}		% def==
\def\A{\mathcal{A}}

\def\id{\mathrm{id}}

\def\coveredby{\lhd}		% Poset covering relation
\usepackage{lipsum}

\begin{document}

%% define your title in the usual way
\title{Palindromic intervals in Bruhat order and hyperplane arrangements}

%% define your authors in the usual way
%% use \addressmark{1}, \addressmark{2} etc for the institutions, and use \thanks{} for contact details
\author{Robert Mcalmon, Suho Oh and Hwanchul Yoo}
\maketitle

\begin{abstract}
An element $w$ of the Weyl group is called rationally smooth if the corresponding Schubert variety is rationally smooth. This happens exactly when the lower interval $[id,w]$ in the Bruhat order is palindromic. For each element $w$ of the Weyl group, we construct a certain hyperplane arrangement. After analyzing the palindromic intervals inside the maximal quotients, we use this result to show that the generating function for regions of the arrangement coincides with the Poincar\'e polynomial of the corresponding Schubert variety if and only if the Schubert variety is rationally smooth.

\end{abstract}

\section{Introduction}

For an element of a Weyl group $w\in W$, let 
$P_w(q) \defeq \sum_{u\leq w} q^{\ell(u)}$, where the sum is over
all elements $u\in W$ below $w$ in the (strong) Bruhat order.
Geometrically, the polynomial $P_w(q)$ is the Poincar\'e polynomial
of the Schubert variety $X_w = BwB/B$ 
in the flag manifold $G/B$.

The inversion hyperplane arrangement $\A_w$ is defined as the collection
of hyperplanes corresponding to all inversions of $w$.  Let $R_w(q) \defeq \sum_r q^{d(r_0,r)}$ 
be the generating function that counts regions $r$ of the arrangement $\A_w$
according to the distance $d(r_0,r)$ from the fixed initial region $r_0$.

The main result of the paper is analyzing the palindromic lower intervals of $W$ and using that to show $P_w(q)=R_w(q)$ if and 
only if the Schubert variety $X_w$ is rationally smooth. We have previously given an elementary combinatorial proof for Type A case of this problem in \cite{OPY08}.

According to the criterion of Carrell and Peterson \cite{C94},
the Schubert variety $X_w$ is rationally smooth if and only if the Poincar\'e polynomial
$P_w(q)$ is palindromic, that is $P_w(q) = q^{\ell(w)} \, P_w(q^{-1})$.
If $w$ is not rationally smooth then the polynomial $P_w(q)$ is not palindromic, 
but the polynomial $R_w(q)$ is always palindromic.  So $P_w(q)\ne R_w(q)$
in this case.  Hence it is enough to show that $P_w(q)=R_w(q)$ when $w$ is rationally smooth. Our proof is purely combinatorial, combining basics of Weyl groups with a result from \cite{BP05}.

This paper is an extended and improved version of the conference paper \cite{OY} back in 2010, with a heavier focus on palindromic lower intervals in parabolic quotients. The previous conjecture of \cite{OY} regarding shapes of palindromic lower intervals will be given a proof using lattice $M(n)$ studied in \cite{St} and \cite{BB05}. This result will be used to prove our main result instead of the original approach, giving us a more uniform proof.

\section*{Acknowledgement}
We would like to thank Alexander Postnikov, Sara Billey, Axel Hultman, Eric Sommers, William Slofstra and Lucas Rusnak for useful remarks and discussions.

\medskip

\section{Rational smoothness of Schubert varieties and Inversion hyperplane arrangement}

In this section we will explain how rational smoothness can be expressed by conditions on the lower Bruhat interval. We will also define the inversion hyperplane arrangement. In this paper, unless stated otherwise, we refer to the strong Bruhat order.

Let $G$ be a semisimple simply-connected complex Lie group, $B$ a Borel subgroup and $\mathfrak{h}$ the corresponding Cartan subalgebra. Let $W$ be the corresponding Weyl group, $\Delta \subset \mathfrak{h}^{*}$ be the set of roots and $\Pi \subset \Delta$ be the set of simple roots. The choice of simple roots determines the set of positive roots. We will write $\alpha > 0$ for $\alpha \in \Delta$ being a positive root. Following the conventions of \cite{BB05}, let $S$ be the set of simple reflections and $T \defeq \{wsw^{-1}:s \in S, w \in W\}$ be the set of reflections. Set $\Pi = \{\alpha_1,\cdots,\alpha_n\}$, $S = \{s_1,\cdots,s_n\}$ and index them properly so that $s_i$ and $\alpha_i$ corresponds to the same node of the Dynkin diagram for $1 \leq i \leq n$. Then there is a bijection between $T$ and $\Delta$ by matching $ws_iw^{-1}$ with $w(\alpha_i)$. Then $ws_iw^{-1}$ is exactly the reflection that reflects by the hyperplane corresponding to the root $w(\alpha_i)$. 

We have the following definitions as in \cite{BB05}:
$$T_L(w) \defeq \{ t \in T : \ell(tw)<\ell(w) \},$$
$$T_R(w) \defeq \{ t \in T : \ell(wt)<\ell(w) \},$$
$$D_L(w) \defeq T_L(w) \cap S,$$ 
$$D_R(w) \defeq T_R(w) \cap S.$$
They are called the left(right) \newword{associated reflections} of $w$ and left(right) \newword{descent set} of $w$. In this paper, we concentrate on lower Bruhat intervals $[\id, w] \defeq \{u\in S_n \mid u\leq w\}$. They are related to \newword{Schubert varieties} $X_w = \overline{BwB/B}$ inside the generalized flag manifold $G/B$. The \newword{Poincar\'e polynomial} of the Schubert variety $X_w$ is the rank generating function for the interval $[\id, w]$,
e.g., see \cite{BL00}:
$$
P_w(q) = \sum_{u\leq w} q^{\ell(u)}.
$$

For convenience, we will say that $P_w(q)$ is the Poincar\'e polynomial of $w$. And we will say that $w$ is rationally smooth if $X_w$ is rationally smooth. Due to Carrell and Peterson, one can check whether the rational locus of a Schubert variety is smooth or not by studying $P_w(q)$. Let us denote a polynomial
$f(q)=a_0 + a_1\, q + \cdots + a_d\, q^d$ 
as \newword{palindromic} if $f(q) = q^d f(q^{-1})$, i.e.,
$a_i = a_{d-i}$ for $i=0,\dots,d$.

\begin{theorem}
\label{th:Peterson_Criterion}
{\rm (Carrell-Peterson~\cite{C94}, see also \cite[Sect.~6.2]{BL00})} 
For any element of a Weyl group $w \in W$, the Schubert variety $X_w$
is rationally smooth if and only if the Poincar\'e polynomial $P_w(q)$
is palindromic.
\end{theorem}

For each $w \in W$, we will be comparing this polynomial $P_w(q)$ with another polynomial, that comes from an associated hyperplane arrangement. To assign a hyperplane arrangement to each $w \in W$, we first need the definition of the inversion set of $w$. The inversion set $\Delta_w$ of $w$ is defined as the following:
$$\Delta_w \defeq \{ \alpha | \alpha \in \Delta, \alpha > 0, w(\alpha) <0 \}.$$

For type A case, this gives the usual definition of an inversion set for permutations. Let us define the \newword{inversion hyperplane arrangement} $\A_w$ as the collection of hyperplanes $\alpha(x)=0$ for all roots $\alpha \in \Delta_w$. Here all the hyperplanes coming from reflections are \newword{central}, meaning that they contain the origin.

Let $r_0$ be the \newword{fundamental chamber} of $\A_w$, the chamber that contains the points satisfying $\alpha(x)>0$ for all $\alpha \in \Delta_w$. Then we can define the \newword{distance enumerating polynomial} on $\A_w$:
$$R_w(q) \defeq \sum_r q^{d(r_0,r)},$$ where the sum
is over all chambers of the arrangement $\A_w$
and $d(r_0,r)$ is the number of hyperplanes separating
$r_0$ and $r$. Our goal in this paper is to show that $R_w(q) = P_w(q)$ whenever $P_w(q)$ is palindromic.

\begin{remark}
We have $P_w(q) = P_{w^{-1}}(q)$ and $R_w(q) = R_{w^{-1}}(q)$ by definition. Whenever we use this fact, we will call this the \newword{duality} of $P_w(q)$ and $R_w(q)$.
\end{remark}

\section{Parabolic Decomposition}

In this section, we introduce a theorem of \cite{BP05} regarding parabolic decomposition that will serve as a key tool in our proof. Let us first recall the definition of the parabolic decomposition. Given a Weyl group $W$, fix a subset $J$ of simple roots. Denote $W_J$ to be the parabolic subgroup generated by simple reflections of $J$. Let $^JW$ be the set of minimal length (right) coset representatives of $W_J \backslash W$. Then it is a well-known fact that every $w \in W$ has a unique parabolic decomposition $w=uv$ where $u \in W_J, v \in {^JW}$ and $\ell(w) = \ell(u) + \ell(v)$.

\begin{lemma}\cite{H74}
For any $w \in W$ and subset $J$ of simple roots, $W_J$ has a unique maximal element below $w$.
\end{lemma}

We will denote the maximal element of $W_J$ below $w$ as $m(w,J)$.

\begin{theorem}
\label{thm:Pfactor}
\cite{BP05}
Let $J$ be any subset of simple roots. Assume $w \in W$ has parabolic decomposition $w=uv$ with $u \in W_J$, $v \in {^JW}$ and $u = m(w,J)$. Then
$$P_w(t) = P_u(t) P_v^{^JW}(t)$$
where $P_v^{^JW} = \sum_{z \in ^JW, z \leq v} t^{\ell(z)}$ is the Poincar\'e polynomial for $v$ in the quotient.
\end{theorem}

This decomposition is very useful in the sense that it allows us to factor the Poincar\'e polynomials. We will say that $J = S \setminus \{\alpha\}$ is \newword{leaf-removed} if $\alpha$ corresponds to a leaf in the Dynkin diagram of $S$.

The following theorem of \cite{BP05} tells us that we only need to look at maximal leaf-removed parabolic subgroups for our purpose.

\begin{theorem}\cite{BP05}\label{thm:BP}
Let $w \in W$ be a rationally smooth element. Then there exists a maximal proper subset $J= S \setminus \{\alpha\}$ of simple roots, such that 
\begin{enumerate}
\item we have a decomposition of $w$ or $w^{-1}$ as in Theorem~\ref{thm:Pfactor},
\item $\alpha$ corresponds to a leaf in the Dynkin diagram of $W$. 
\end{enumerate}
\end{theorem}

We will call the parabolic decompositions that satisfies the conditions of the above theorem as \newword{BP-decompositions}. This is a strong tool and has been lead to numerous interesting results on Schubert varieties \cite{AR},\cite{KOO},\cite{Richmond2014}, \cite{Richmond2016},\cite{SLOF},.

\section{Palindromic lower intervals in $W^J$}

In this section, we will describe the palindromic lower intervals of $^JW$, where $J = S \setminus \{\alpha\}$ is leaf-removed as before. Due to the obvious symmetry, we will work with left quotient $W^J$ instead since lot of the tools in the previous literature is stated for $W^J$ instead. We say that $v$ is \newword{trivial} if $v=id$ or $v$ is the longest element of $W^J$. If the lower interval $[id,v]$ in $W^J$ is a chain poset, we will say that $v$ is a \newword{chain element} of $W^J$. Whenever $[id,v]$ is rank-symmetric, we will say that the lower interval of $v$ is palindromic or simply say that $v$ is palindromic.

Let $I$ be the set of simple roots that appear in a reduced word of $v$. We say that $v$ is a \newword{locally-longest} element in $W^J$ if it is the longest element of $W_I^{I \cap J}$ and $I$ forms a connected subgraph within the Dynkin diagram. Such quotients of form $W_I^{I \cap J}$ will be referred as \newword{embedded quotients}. Similarly we will say that $v$ is in a \newword{local chain} if $W_I^{I \cap J}$ is a chain poset. In this section, we will show that a palindromic element $v$ is either a locally-longest element or is in a local chain except for two cases of $v$.

\begin{remark}
In \cite{SLOF}, the main tool for decomposing the polynomials was studying the chain elements (called chain BP-decomposition), which is when $v$ is a chain element. In this paper, our decomposition focuses on dealing with the locally-longest elements instead.
\end{remark}

\subsection{type A,F and G}

\begin{example}[$A_n$]
Choose $A_n$ to be our choice of Weyl group and label the simple roots $S = \{s_1,\cdots,s_n\}$ so that the labels match the corresponding nodes in the Dynkin diagram of Figure~\ref{dia:An}. If we set $J = S \setminus \{s_1\}$ then the list of nontrivial palindromic $v \in W^J$ are:
$$s_1,s_2s_1,\ldots,s_{n-1}\cdots s_1.$$

They are all locally-longest elements. Same happens for $J = S \setminus \{s_n\}$ as well.

\end{example}

\begin{figure}[h!]
%\label{dia:An}
\centering\begin{tikzpicture}
\draw[line width=1pt]
	(0,0)--(2,0)
	(3,0)--(5,0);
\draw[loosely dotted, thick]
	(2,0)--(3,0);
\draw[fill=white, line width=1pt]
	(0,0) circle[radius=1.8mm]
	(1,0) circle[radius=1.8mm]
	(4,0) circle[radius=1.8mm]
	(5,0) circle[radius=1.8mm];
\node at (0,-.5) {$s_1$};
\node at (1,-.5) {$s_2$};
\node at (4,-.5) {$s_{n-1}$};
\node at (5,-.5) {$s_n$};
\end{tikzpicture}
\caption{Type $A_n$ Dynkin diagram}\label{dia:An}
\end{figure}

\begin{example}[$F_4$]
\label{ex:F4}
Choose $F_4$ to be our choice of Weyl group and label the simple roots $S = \{s_1,\cdots,s_4\}$ so that the labels match the corresponding nodes in the Dynkin diagram of Figure~\ref{dia:F4}. If we set $J = S \setminus \{s_4\}$, then the list of nontrivial palindromic $v \in W^J$ are:
$$s_4, s_3 s_4, s_2 s_3 s_4, s_1 s_2 s_3 s_4, s_3 s_2 s_3 s_4, s_4 s_3 s_2 s_3 s_4.$$
% s_4 s_3 s_2 s_3 s_1 s_2 s_3 s_4 s_3 s_2 s_3 s_1 s_2 s_3 s_4.$$

%$$s_4, s_4 s_3, s_4 s_3 s_2, s_4 s_3 s_2 s_1, s_4 s_3 s_2 s_3, s_4 s_3 s_2 s_3 s_4, s_4 s_3 s_2 s_3 s_1 s_2 s_3 s_4 s_3 s_2 s_3 s_1 s_2 s_3 s_4.$$

The elements that are not locally-longest are $s_2 s_3 s_4, s_3 s_2 s_3 s_4$ and  $s_1 s_2 s_3 s_4$. The elements $s_2s_3s_4$ and $s_3s_2s_3s_4$ are contained in the same local chain. By symmetry, $J = S \setminus \{s_1\}$ case is similar.

\end{example}

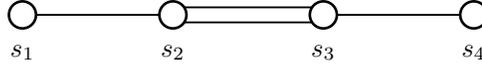
\begin{figure}[h!]
\centering\begin{tikzpicture}
\draw[line width=0.7pt]
	(0,0)--(2,0);
\draw[line width=0.7pt]
	(4,0)--(6,0);

\draw[line width=0.7pt]
	(2,0.1)--(4,0.1);
\draw[line width=0.7pt]
	(2,-0.1)--(4,-0.1);

\draw[fill=white, line width=1pt]
	(0,0) circle[radius=1.8mm]
	(2,0) circle[radius=1.8mm]
	(4,0) circle[radius=1.8mm]
	(6,0) circle[radius=1.8mm];
\node at (0,-.5) {$s_1$};
\node at (2,-.5) {$s_2$};
\node at (4,-.5) {$s_3$};
\node at (6,-.5) {$s_4$};
%\node at (3,.5) {$4$};
\end{tikzpicture}
\caption{Type $F_4$ Dynkin diagram}\label{dia:F4}
\end{figure}

\begin{example}[$G_2$]
Choose $G_2$ to be our choice of Weyl group and label the simple roots $S = \{s_1,s_2\}$ so that the labels match the corresponding nodes in the Dynkin diagram of Figure~\ref{dia:G2}. If we set $J = S \setminus \{s_1\}$, then the list of nontrivial palindromic $v \in W^J$ are:
$$s_1,s_2s_1,s_1s_2s_1,\ldots.$$
Only $s_1$ is a locally-longest element here. But $W^J$ is a chain itself so we can say all other elements are in a local chain. Same can be said for $J = S \setminus \{s_2\}$.

\end{example}

\begin{figure}[h!]
\centering\begin{tikzpicture}
%\draw[line width=0.7pt]
%	(0,0)--(2,0);
%\draw[line width=0.7pt]
%	(4,0)--(6,0);

\draw[line width=0.7pt]
	(2,0.1)--(4,0.1);
\draw[line width=0.7pt]
	(2,0.0)--(4,0.0);

\draw[line width=0.7pt]
	(2,-0.1)--(4,-0.1);

\draw[fill=white, line width=1pt]
	%(0,0) circle[radius=1.8mm]
	(2,0) circle[radius=1.8mm]
	(4,0) circle[radius=1.8mm];
	%(6,0) circle[radius=1.8mm];
%\node at (0,-.5) {$s_1$};
\node at (2,-.5) {$s_1$};
\node at (4,-.5) {$s_2$};
%\node at (6,-.5) {$s_4$};
%\node at (3,.5) {$4$};
\end{tikzpicture}
\caption{Type $G_2$ Dynkin diagram}\label{dia:G2}
\end{figure}
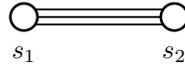

\subsection{type B and D}

Define $M(n)$ to be the set of subsets of $[n] \defeq \{1,\ldots,n\}$ endowed with the partial ordering $\preceq$ defined as follows:
let $A,B\subseteq[n]$.
Write $A=\{a_1<\dots<a_j\}$ and $B=\{b_1<\dots<b_k\}$.
Then $A\preceq B$ denotes that $j\leq k$ and $a_{j-i}\leq b_{k-i}$ for every $0\leq i\leq j-1$.

\begin{figure}[h!]
\centering
\begin{tikzpicture}
\node (e) at (0,0) {$\varnothing$};
\node (s0) at (0,1) {$1$};
\node (s10) at (0,2) {$2$};
\node (s010) at (-1,3) {$12$};
\node (s210) at (1,3) {$3$};
\node (s2010) at (0,4) {$13$};
\node (s3210) at (2,4) {$4$};
\node (s12010) at (-1,5) {$23$};
\node (s32010) at (1,5) {$14$};
\node (s012010) at (-2,6) {$123$};
\node (s312010) at (0,6) {$24$};
\node (s3012010) at (-1,7) {$124$};
\node (s2312010) at (1,7) {$34$};
\node (s23012010) at (0,8) {$134$};
\node (s123012010) at (0,9) {$234$};
\node (s0123012010) at (0,10) {$1234$};
\draw (e) -- (s0) -- (s10) -- (s010) -- (s2010) -- (s12010) -- (s012010) -- (s3012010) -- (s23012010) -- (s123012010) -- (s0123012010);
\draw (s10) -- (s210) -- (s3210) -- (s32010) -- (s312010) -- (s2312010) -- (s23012010);
\draw (s210) -- (s2010) -- (s32010);
\draw (s12010) -- (s312010) -- (s3012010);
\end{tikzpicture}
\caption{The lattice $M(4)\cong B_4/A_3$}
\label{F:m4}
\end{figure}
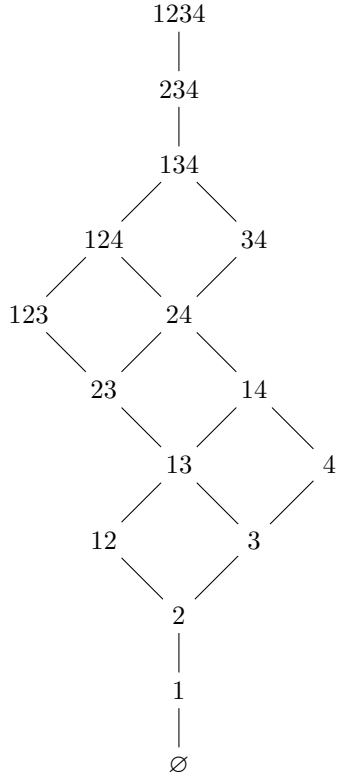

\begin{proposition}
\label{prop:mnposet}
The palindromic elements of $M(n)$ are $k$ and $[k]$ for positive integers $k \leq n$ together with $\emptyset$.

%and $[k]$ for positive integers $k \leq n$.
\end{proposition}
\begin{proof}
For $A \in M(n)$, define
\begin{eqnarray}\label{eq:updown}
U(A) \defeq \{B\in M(n) \colon A\coveredby B\},\\
%D(A) \defeq \{B\in M(n) \colon B \coveredby A\}
U^2(A) \defeq \{B \in M(n) \colon \exists C,  A \coveredby C \coveredby B \},
\end{eqnarray}
where $A \coveredby B$ stands for the covering relation: $A \preceq B$ and there is no other $C$ such that $A \preceq C \preceq B$. Sets $D(A)$ and $D^2(A)$ are defined similarly for checking elements below $A$ instead. It is fairly simple to verify that $[\emptyset,k]$ and $[\emptyset,[k]]$ are palindromic.

Assume for the sake of contradiction that there is a subset $A\subset[n]$ such that $[\varnothing,A]$ is a palindromic interval, and so that $A\neq k,[k]$ for any $k\leq n$.
We take $n\geq 3$, since $M(2)\cong[3]$. Examine the bottom ranks of $M(n)$, we get $|U(\varnothing)| = |U^2(\varnothing)| = 1$. For $A$ to be a palindromic element, we need to have $|D(A)| = |D^2(A)| = 1$ as well. Since $|D(A)| = 1$, then $A$ must be a succession of positive integers $A = \{j,j+1,\dots,k-1,k\}$ for some $1<j<k\leq n$. Since $A = \{j,\dots,k\}$ with $j>1$, $D(A) = \{\{j-1,j+1,\dots,k\}\}$ and we get $|D^2(A)| >1$. This contradicts the assumption that $[\varnothing,A]$ is palindromic.

\end{proof}

\begin{example}
\label{ex:m4}\rm
Look at the poset $M(4)=[\varnothing,1234]$ in Figure~\ref{F:m4}.
Notice how $M(1)=[\varnothing,1]$, $M(2)=[\varnothing,12]$ and $M(3)=[\varnothing,123]$.
In fact, this chain continues on infinitely as $M(1)\subset M(2)\subset M(3)\subset\cdots$, and we may consider every $M(n)$ as living inside the infinite poset $M(\infty)$, which is the set of positive integers $\mathbb N$ endowed with the same partial order relation $\preceq$ described in this section.
The only elements with non-chain, palindromic lower intervals are $1234$, and $123$.

\end{example}

We now start analyzing the type B case. Figure~\ref{dia:Bn} is a Dynkin diagram of type B. 

\begin{figure}[h!]
\centering\begin{tikzpicture}
\draw[line width=1pt]
  (0,0.1)--(1,0.1)
	(0,-0.1)--(1,-0.1);
\draw[line width=1pt]
	(1,0)--(3,0)
	(4,0)--(6,0);
\draw[loosely dotted, thick]
	(3,0)--(4,0);
\draw[fill=white, line width=1pt]
	(0,0) circle[radius=1.8mm]
	(1,0) circle[radius=1.8mm]
	(2,0) circle[radius=1.8mm]
	(5,0) circle[radius=1.8mm]
	(6,0) circle[radius=1.8mm];
\node at (0,-.5) {$s_0$};
%\node at (.5,.5) {4};
\node at (1,-.5) {$s_1$};
\node at (2,-.5) {$s_2$};
\node at (5,-.5) {$s_{n-2}$};
\node at (6,-.5) {$s_{n-1}$};
\end{tikzpicture}
\caption{Type $B_n$ Dynkin diagram}\label{F:bn}
\label{dia:Bn}
\end{figure}
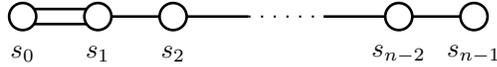

\begin{lemma}[$B_n/B_{n-1}$]
\label{lem:bnrightleft}
Let $W$ be a type $B_n$ Weyl group and $S$ be the set of simple roots, corresponding to the Dynkin diagram of Figure~\ref{dia:Bn}. Set $J=S\setminus\{s_{n-1}\}$. Then $W^J$ is a chain poset of length $2n$.
\end{lemma}
\begin{proof}
Every element of $W^J$ can be written as a tail of the reduced expression of the longest element
$v_0 = s_{n-1} \cdots s_1 s_0 s_1 \cdots s_{n-1}$.
\end{proof}

An example of $W^J$ in this case is shown in Figure~\ref{F:quotientb3b2}.
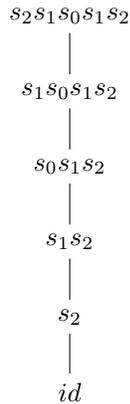
\begin{figure}[h!]
\centering
\begin{tikzpicture}
\node (id) at (0,0) {$id$};
\node (s2) at (0,1) {$s_2$};
\node (s12) at (0,2) {$s_1s_2$};
\node (s012) at (0,3) {$s_0s_1s_2$};
\node (s1012) at (0,4) {$s_1s_0s_1s_2$};
\node (s21012) at (0,5) {$s_2s_1s_0s_1s_2$};
\draw (id) -- (s2) -- (s12) -- (s012) -- (s1012) -- (s21012);
\end{tikzpicture}
\caption {Type $B_3/B_2$ quotient}
\label{F:quotientb3b2}
\end{figure}

The following Lemma is from exercise~$6$ in Chapter $8$ of \cite{BB05}:

\begin{lemma}[$B_n/A_{n-1}$, \cite{BB05}]
\label{lem:bnisomn}
Let $W$ be a type $B_n$ Weyl group and $S$ be the set of simple roots, corresponding to the Dynkin diagram of Figure~\ref{dia:Bn}. Set $J=S\setminus\{s_0\}$. Then $W^J \cong M(n)$. 
\end{lemma}

\begin{proposition}
\label{cor:bnpalin}
The palindromic elements of $W^J$(in the previous Lemma) are exactly the locally-longest elements and chain elements of form $s_k \cdots s_0$ for $0 \leq k \leq n-1$.
\end{proposition}
\begin{proof}
From Lemma~\ref{lem:bnisomn} and Proposition~\ref{prop:mnposet}, the palindromic elements are pre-images of $\emptyset, 1,\ldots,n$ and  $[2], \ldots, [n]$. There are total $2n$ of them. We have $n+1$ locally-longest elements coming from $I = \emptyset, \{0\},\{0,1\},\ldots,\{0,\ldots,n-1\}$ and $n+1$ chain elements coming from $id, s_0,s_1s_0,\ldots,s_{n-1}\cdots s_0$. Since we have $2$ overlaps, these $2n$ elements are exactly all the palindromic elements of $W^J$.

\end{proof}

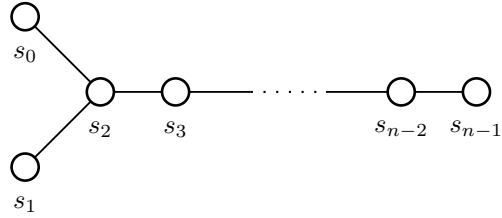
\begin{figure}[h!]
\centering\begin{tikzpicture}
\draw[line width=0.7pt]
	(0,1)--(1,0)
	(0,-1)--(1,0)
	(1,0)--(3,0)
	(4,0)--(6,0);
\draw[loosely dotted, thick]
	(3,0)--(4,0);
\draw[fill=white, line width=1pt]
	(0,1) circle[radius=1.8mm]
	(0,-1) circle[radius=1.8mm]
	(1,0) circle[radius=1.8mm]
	(2,0) circle[radius=1.8mm]
	(5,0) circle[radius=1.8mm]
	(6,0) circle[radius=1.8mm];
\node at (0,.5) {$s_0$};
\node at (0,-1.5) {$s_1$};
\node at (1,-.5) {$s_2$};
\node at (2,-.5) {$s_3$};
\node at (5,-.5) {$s_{n-2}$};
\node at (6,-.5) {$s_{n-1}$};
\end{tikzpicture}
\caption{Type $D_n$ Dynkin diagram}\label{dia:Dn}
\end{figure}

We now start analyzing the type D case. Figure~\ref{dia:Dn} is a Dynkin diagram of type D. We call the following leaf-removed quotient a type $D_n/D_{n-1}$ quotient.

\begin{lemma}[$D_n/D_{n-1}$]
\label{lem:typednpalin1}\rm
Let $W$ be a Weyl group of type $D_n$ with $n\geq 4$ and $S$ be the set of simple roots, corresponding to the Dynkin diagram of Figure~\ref{dia:Dn}. Let $J=S\setminus\{s_{n-1}\}$. Every palindromic element $v \in W^J$ is a locally-longest element.
\end{lemma}
\begin{proof}
Every element having only $s_{n-1}$ as a right descent can be written as a right tail of one of two reduced expressions of the maximal element:
$s_{n-1}\dots s_2 s_1 s_0 s_2 \dots s_{n-1} = s_{n-1}\dots s_2 s_0 s_1 s_2 \dots s_{n-1}$. 

Since $s_1$ and $s_0$ commutes, we have exactly one element of length $k$ for $0\leq k\leq 2n-2$ and $k\neq n-1$, and exactly two elements of length $n-1$. An example Hasse diagram is drawn in Figure~\ref{F:d5d4}. The nontrivial palindromic elements are exactly the elements of length $1 \leq k \leq n-1$.

\end{proof}

\begin{figure}
\centering
\begin{tikzpicture}
\node (e) at (0,0) {$e$};
\node (s4) at (0,1) {$s_4$};
\node (s34) at (0,2) {$s_3s_4$};
\node (s234) at (0,3) {$s_2s_3s_4$};
\node (s0234) at (-1,4) {$s_0s_2s_3s_4$};
\node (s1234) at (1,4) {$s_1s_2s_3s_4$};
\node (s10234) at (0,5) {$s_1s_0s_2s_3s_4$};
\node (s210234) at (0,6) {$s_2s_1s_0s_2s_3s_4$};
\node (s3210234) at (0,7) {$s_3s_2s_1s_0s_2s_3s_4$};
\node (s43210234) at (0,8) {$s_4s_3s_2s_1s_0s_2s_3s_4$};
\draw (e) -- (s4) -- (s34) -- (s234) -- (s0234) -- (s10234) -- (s210234) -- (s3210234) -- (s3210234) -- (s43210234);
\draw (s234) -- (s1234) -- (s10234);
\end{tikzpicture}
\caption{Type $D_5/D_4$ quotient}\label{F:d5d4}
\end{figure}
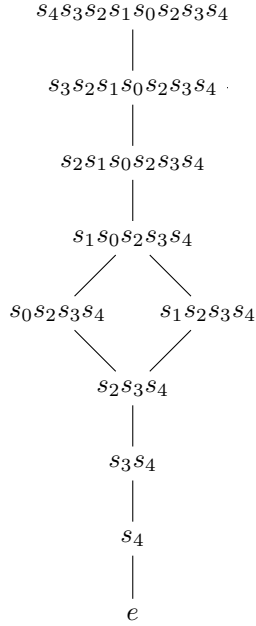

For $J=S\setminus\{s_0\}$ and $J'=S\setminus\{s_1\}$, we call $W^J\cong W^{J'}$ a type $D_n/A_{n-1}$ quotient. This quotient is isomorphic to $M(n-1)$ thanks to Stanley \cite{St}:

\begin{lemma}[$D_n/A_{n-1}$,\cite{St}]
\label{lem:dnmn}\
Let $W$ be a Weyl group of type $D_n$ for $n\geq 4$.
Let $J=S\setminus\{s_0\}$ or $J=S\setminus\{s_1\}$.
Then $W^J\cong M(n-1)$. 
\end{lemma}

\begin{proposition}
\label{prop:dnpalin}
The palindromic elements of $W^J$(in the previous Lemma) are exactly the locally-longest elements.
\end{proposition}
\begin{proof}
From Lemma~\ref{lem:dnmn} and Proposition~\ref{prop:mnposet}, the palindromic elements are pre-images of $\emptyset,1,\ldots,n-1$ and $[2], \ldots, [n-1]$. There are total $2n-2$ of them. We have $2n-2$ locally-longest elements coming from $I = \emptyset,\{0\},\{0,2\},\ldots,\{0,2,3,\ldots,n-1\}, \{0,2,1\},\ldots,\{0,2,1,3,\ldots,n-1\}$. Hence these are exactly all of the palindromic elements of $W^J$.

\end{proof}

\subsection{type E}

We now start analyzing the type E case. Figure~\ref{dia:En} has Dynkin diagrams of type $E_6,E_7$ and $E_8$. Since the leaf quotients of $E_6$ and $E_7$ are embedded inside the leaf quotients of $E_8$, we will analyze $E_8$ only.

\begin{figure}[h!]
\centering
\subfigure[$E_6$]{
\begin{tikzpicture}
\draw[line width=0.7pt]
	(0,0)--(4,0)
	(2,0)--(2,1);
\draw[fill=white, line width=1pt]
	(0,0) circle[radius=1.8mm]
	(1,0) circle[radius=1.8mm]
	(2,0) circle[radius=1.8mm]
	(2,1) circle[radius=1.8mm]
	(3,0) circle[radius=1.8mm]
	(4,0) circle[radius=1.8mm];
\node at (0,-.5) {$s_1$};
\node at (1,-.5) {$s_3$};
\node at (2,-.5) {$s_4$};
\node at (2.5,1) {$s_2$};
\node at (3,-.5) {$s_5$};
\node at (4,-.5) {$s_6$};
\end{tikzpicture}
}\qquad
\subfigure[$E_7$]
{
\begin{tikzpicture}
\draw[line width=0.7pt]
	(0,0)--(5,0)
	(2,0)--(2,1);
\draw[fill=white, line width=1pt]
	(0,0) circle[radius=1.8mm]
	(1,0) circle[radius=1.8mm]
	(2,0) circle[radius=1.8mm]
	(2,1) circle[radius=1.8mm]
	(3,0) circle[radius=1.8mm]
	(4,0) circle[radius=1.8mm]
	(5,0) circle[radius=1.8mm];
\node at (0,-.5) {$s_1$};
\node at (1,-.5) {$s_3$};
\node at (2,-.5) {$s_4$};
\node at (2.5,1) {$s_2$};
\node at (3,-.5) {$s_5$};
\node at (4,-.5) {$s_6$};
\node at (5,-.5) {$s_7$};
\end{tikzpicture}
}\\
\subfigure[$E_8$]
{
\begin{tikzpicture}
\draw[line width=0.7pt]
	(0,0)--(6,0)
	(2,0)--(2,1);
\draw[fill=white, line width=1pt]
	(0,0) circle[radius=1.8mm]
	(1,0) circle[radius=1.8mm]
	(2,0) circle[radius=1.8mm]
	(2,1) circle[radius=1.8mm]
	(3,0) circle[radius=1.8mm]
	(4,0) circle[radius=1.8mm]
	(5,0) circle[radius=1.8mm]
	(6,0) circle[radius=1.8mm];
\node at (0,-.5) {$s_1$};
\node at (1,-.5) {$s_3$};
\node at (2,-.5) {$s_4$};
\node at (2.5,1) {$s_2$};
\node at (3,-.5) {$s_5$};
\node at (4,-.5) {$s_6$};
\node at (5,-.5) {$s_7$};
\node at (6,-.5) {$s_8$};
\end{tikzpicture}
}
\caption{Finite type $E_6,E_7,E_8$}\label{dia:En}
\end{figure}
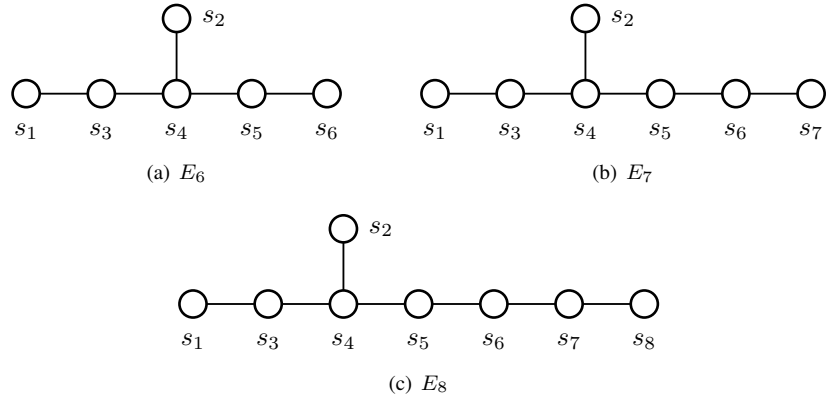

\begin{example}[$E_8$]
\label{ex:E8}
The following are the elements of $^{S \setminus \{s\}}W$ when $W$ is a Weyl group of exceptional type, $S$ the simple roots of $W$ and $s$ a leaf of the Dynkin diagram as in Figure~\ref{dia:En}. We only list the nontrivial palindromic elements. The result was obtained by using the Coxeter package developed by Stembridge. For this example, we use $s_i$ and $i$ interchangeably, so a list of numbers corresponds to a word of simple roots.

\begin{enumerate}
\item $W=E_8, s=2$
\begin{itemize}
\item $[2]$
\item $[2,4]$
\item $[2,4,3]$
\item $[2,4,3,1]$
\item $[2,4,3,1,5,6,4,5,3,4,2,4,3,1,5,6,4,5,3,4,2]$
\item $[2,4,5]$
\item $[2,4,5,6]$
\item $[2,4,5,6,7]$
\item $[2,4,5,6,7,8]$
\item $[2,4,5,3,4,2]$
\item $[2, 4, 5, 3, 4, 2, 6, 5, 4, 3]$
\item $[2, 4, 5, 3, 4, 2, 6, 7, 5, 6, 4, 5, 3, 4, 2]$
\item $[2, 4, 5, 3, 4, 2, 6, 7, 5, 6, 4, 5, 3, 4, 2, 8, 7, 6, 5, 4, 3]$
\item $[2, 4, 5, 3, 4, 2, 1, 3, 4, 5]$
\item $[2, 4, 5, 3, 4, 2, 1, 6, 7, 5, 6, 4, 5, 3, 4, 2, 5, 4, 3, 1, 6, 7, 5, 6, 4, 5, 3, 4, 2, 5, 4, 3, 1, 6, 7, 5, 6, 4, 5, 3, 4, 2]$
\end{itemize}

\item $W=E_8,s=1$
\begin{itemize}
\item $[1]$
\item $[1,3]$
\item $[1,3,4]$
\item $[1,3,4,5]$
\item $[1,3,4,5,6]$
\item $[1,3,4,5,6,7]$
\item $[1,3,4,5,6,7,8]$
\item $[1,3,4,2]$
\item $[1, 3, 4, 5, 2, 4, 3, 1]$
\item $[1, 3, 4, 5, 2, 4, 3, 1, 6, 5, 4, 3, 2, 4, 5, 6]$
\item $[1, 3, 4, 5, 6, 7, 2, 4, 5, 6, 3, 4, 5, 2, 4, 3, 1, 3, 4, 5, 6, 7, 2, 4, 5, 6, 3, 4, 5, 2, 4, 3, 1]$
\end{itemize}

\item $W=E_8,s=8$
\begin{itemize}
\item $[8]$
\item $[8,7]$
\item $[8,7,6]$
\item $[8,7,6,5]$
\item $[8,7,6,5,4]$
\item $[8,7,6,5,4,2]$
\item $[8,7,6,5,4,3]$
\item $[8,7,6,5,4,3,1]$
\item $[8, 7, 6, 5, 4, 3, 2, 4, 5, 6, 7, 8]$
\end{itemize}

In each of the cases, every element is a locally-longest element.

\end{enumerate}

\end{example}

\subsection{Conclusion}

Combining the case-by-case results analyzed so far, we reach the following conclusion:

\begin{theorem}
\label{thm:quotpalin}
Let $W$ be a Weyl group with simple reflections $S$, and let $J \defeq S \setminus \{s\}$ where $s$ is a leaf of the Dynkin diagram. Then $v \in ^JW$ is palindromic if and only if one of the following holds:
% $v$ is a locally-longest element or is in a local-chain except for the following cases:
\begin{itemize}
\item $v$ is a locally-longest element,
\item $v$ is in a local-chain,
\item Case $F_4$ with $J = S \setminus \{s_4\}$ and $v$ as $s_4 s_3 s_2 s_1$,
\item Case $F_4$ with $J = S \setminus \{s_1\}$ and $v$ as $s_1 s_2 s_3 s_4$,
\item Case $B_n$ with $J = S \setminus \{s_0\}$ and $v$ as $s_0 \cdots s_k$ for some $0 \leq k \leq n-1$.
\end{itemize}

\end{theorem}

This implies that a rationally smooth $w$ has a BP-decomposition where $v$ is a locally-longest element or is in a local chain except for few cases. Similarly, it was shown in \cite{SLOF} that a rationally smooth $w$ has a BP-decomposition into a chain element except for few special cases.

The theorem implies a conjecture that was previously asked by the authors \cite{OY}:

\begin{corollary}
\label{cor:conj}
Let $W$ be a Weyl group with simple reflections $S$, and let $J \defeq S \setminus \{s\}$ where $s$ is a leaf of the Dynkin diagram. Then $v \in W^J$ is palindromic if and only if $v$ is a locally-longest element or a chain element.
\end{corollary}

The conjecture was actually stated for entire Coxeter groups, but we show in the last section that it is not true for type $H_4$.

\section{Factoring the polynomial $R_w(q)$}
In this section we will study how $R_w(q)$ behaves with respect to the BP-decomposition. From a hyperplane arrangement $\A$, we can get a poset structure on the set of chambers of that arrangement.

\begin{definition}
Given an arrangement $\A$ and a choice of its fundamental chamber $r_0$, we get a poset structure $Q_{\A}$ on the set of chambers by the covering relation $r \triangleleft r'$ whenever $r$ and $r'$ are adjacent (separated by a hyperplane) and $d(r_0,r') = d(r_0,r)+1$. 
\end{definition}

Then the distance enumerating polynomial $R_{\A}$ of $\A$ is simply the rank generating function of its poset $Q_{\A}$. 

Let $\A$ be a central hyperplane arrangement with a pre-fixed fundamental chamber $r_0$. Also let $\A'$ be some subarrangement of $\A$ and $r$ some chamber of $\A'$. We define the induced subposet $Q_{\A,\A',r}$ to be the induced subposet of $Q_{\A}$ on the chambers of $\A$ that is contained in $r$. We will say that $\A$ is \newword{uniform} with respect to $\A'$ if for all chambers $r$ of $\A'$, the induced subposets $Q_{\A,\A',r}$ are all isomorphic. When this happens, we use $Q_{\A,\A'}$ to denote the poset.

If $w_0$ is the longest element of $W$, the arrangement $\A_{w_0}$ is simply the well known \newword{Coxeter arrangement} of $W$. Here each chamber is indexed with a permutation $w \in W$ and two chambers $u,v$ are adjacent if and only if $v=us_i$ and their length differs exactly by $1$. Hence the poset $Q_{\A_{w_0}}$ where $w_0$ is the longest element of $W$ is exactly the \newword{weak Bruhat order} of $W$. It is a well known fact that the weak Bruhat order of $W$ and the strong Bruhat order of $W$ are different poset structures on the same set of elements with same rank \cite{BB05}. From this it follows that:

\begin{lemma}
\label{lem:maxsame}
When $w_0$ is the longest element of $W$, we have $P_{w_0}(q) = R_{w_0}(q)$. When $u_0$ is the longest element of $W_J$ for some $J \subset S$, $\A_{w_0}$ is uniform with respect to $\A_{u_0}$. %Moreover, the poset $Q_{\A_{w_0},\A_{u_0}}$ is a chain.
\end{lemma}
\begin{proof}
Each chamber of $\A_{w_0}$ is indexed with a permutation of $w \in W$. Each permutation $w \in W$ has a parabolic decomposition $uv$ where $u \in W_J$ and $v \in ^JW$. The chambers indexed by $uv$ with common $u \in W_J$ are contained in the same chamber indexed by $u$ in $\A_{u_0}$. For each chamber $u$ in $\A_{u_0}$, the chambers of $\A_{w_0}$ contained in it are only separated by hyperplanes of $\A_{w_0} \setminus \A_{u_0}$. The poset $Q_{\A_{w_0},\A_{u_0}}$ is the right weak Bruhat order on $W^J$.
\end{proof}

Now we start the analysis of $R_w(q)$ when $w$ is rationally smooth. The first step is the following lemma:

\begin{lemma}\label{u:right descent}
Let $w\in W$ be a rationally smooth element and $w=uv$ be a BP-decomposition. Then every simple reflection in $J$ appearing in the reduced word of $v$
is a right descent of $u$.
\end{lemma}

\begin{proof}
 If we delete every simple reflection appearing in $v$ but one in $J$, then the resulting element is in $W_J$ and is below $w$. Hence by maximality of $u$, it is below $u$.
\end{proof}

Actually, we can state much more about $u$ in terms of simple reflections of $J$ appearing in $v$.

\begin{lemma}\label{u:right reflection}
Let $w=uv$ be a BP-decomposition with respect to $J$. Let $I$ be the subset of $S$ that appears in the reduced word of $v$. Then every reflection formed by simple reflections in $I\cap J$ is a right inversion reflection of $u$. In fact, there is a minimal length decomposition $u=u'u_{I \cap J}$ where $u_{I \cap J}$ is the longest element of $W_{I\cap J}$.
\end{lemma}

\begin{proof}
Take the parabolic decomposition of $u$ under the right quotient by $W_{I\cap J}$. Say, $u=u'u''$. Then $u'$ is the minimal length representative of $u$ in $W/W_{I\cap J}$. For any simple reflection $s\in I\cap J$, the minimal length representative of $us$ in $W/W_{I\cap J}$ is still $u'$, hence the parabolic decomposition
of $us$ is $us=u'(u''s)$. Since $s$ is a right descent of $u$ by Lemma~\ref{u:right descent}, $s$ is a right descent of $u''$. Therefore $u''$ is the longest element in $W_{I\cap J}$. The rest follows from this.
\end{proof}

The above lemma tells us that for each rationally smooth $w \in W$, we can decompose $w$ or $w^{-1}$ to $u' u_{I \cap J} v$ where $uv$ is the BP-decomposition with respect to $J$, with $u = u'u_{I \cap J}$ and $u_{I \cap J}$ is the longest element of $W_{I \cap J}$. Given such decomposition, we decompose $\A_w$ into $\A_0 \defeq \A_{u_{I \cap J}}, \A_1 \defeq \A_u \setminus \A_0$ and $\A_2 \defeq \A_w \setminus \A_u$.

\begin{proposition}
\label{prop:chamred}
Let $r$ be some chamber inside $\A_1 \sqcup \A_0$. Let $r'$ be the chamber of $\A_0$ that contains $r$. Then the poset $Q_{\A_w, \A_1 \sqcup \A_0, r}$ is isomorphic to $Q_{\A_0 \sqcup \A_2, \A_0, r'}$.
\end{proposition}
\begin{proof}

Once a chamber $r'$ of $\A_0$ is fixed, we will show that any chamber of $\A_0 \sqcup \A_2$ contained in $r'$ interesects every chamber of $\A_1 \sqcup \A_0$ contained in $r'$. In order to show this we can freely add more hyperplanes to $\A_0, \A_1$ and $\A_2$. So we may assume that $u = u' u_{I \cap J}$ is the longest element of $W_J$ and $v$ is the longest element of $^JW$. 

From Lemma~\ref{lem:maxsame}, each chamber of $\A_0$ is now indexed with a permutation of $W_{I \cap J}$. Fix a chamber $r_x$ labeled with a permutation $x \in W_{I \cap J}$. Each chamber of $\A_0 \sqcup \A_2$ contained in $r_x$ is labeled with a permutation $xz$ where $z \in ^JW$. Each chamber of $\A_1 \sqcup \A_0$ contained in $r_x$ is labeled with a permutation $yx$ where $y \in W_J^{I \cap J}$. For any such chamber of $\A_0 \sqcup \A_2$ and $\A_1 \sqcup \A_0$, their intersection will be the chamber of $\A$ that is labeled by $yxz$, a permutation of $W$. 

Let $r_1$ and $r_2$ be two different chambers of $\A$ contained in $r$. They are separated by a hyperplane in $\A_2$. For $i=1,2$, let $r_i'$ be the chamber of $\A_0 \sqcup \A_2$ that contains $r_i$. Then $r_1'$ and $r_2'$ are different chambers since they are separated by the hyperplane that separates $r_1$ and $r_2$. If $r_1$ and $r_2$ are adjacent, then $r_1'$ and $r_2'$ are adjacent. If $r_1'$ and $r_2'$ are adjacent but $r_1$ and $r_2$ are not, that means there is a hyperplane of $\A_1$ that separates $r_1$ and $r_2$. But that contradicts the fact that $r_1$ and $r_2$ are both contained in the same chamber of $\A_1 \sqcup \A_0$. So we may conclude that $r_1$ and $r_2$ are adjacent if and only if $r_1'$ and $r_2'$ are. 

\end{proof}

From the above property we immediately get the following tool:
\begin{corollary}
\label{cor:main}
In the above decomposition, $\A_{u_{I \cap J}v}$ being uniform with respect to $\A_{u_{I \cap J}}$, on top of $R_{u_{I \cap J}v}(q) = P_{u_{I \cap J}v}(q)$ and $R_u(q) = P_u(q)$ implies $R_w(q) = P_w(q)$.
\end{corollary}
\begin{proof}

If we know that $\A_{u_{I \cap J}v}$ is uniform with respect to $\A_{u_{I \cap J}}$, then Proposition~\ref{prop:chamred} tells us that $\A_w$ is uniform with respect to $\A_u$. Hence $R_w(q)$ is divisible by $R_u(q)$. Moreover, $\frac{R_w(q)}{R_u(q)} = \frac{R_{u_{I \cap J}v}(q)}{R_{u_{I \cap J}}(q)}$. From Lemma~\ref{lem:maxsame} we have $R_{u_{I \cap J}}(q) = P_{u_{I \cap J}}(q)$. Hence $R_{u_{I \cap J}v}(q) = P_{u_{I \cap J}v}(q)$ and $R_u(q) = P_u(q)$ would imply $R_w(q) = P_w(q)$.
%In the above decomposition $\frac{R_w(q)}{R_u(q)} = \frac{R_{u_{I \cap J}v}(q)}{R_{u_{I \cap J}}(q)}$. We also have $R_{u_{I \cap J}}(q) = P_{u_{I \cap J}}(q)$ from Lemma~\ref{lem:maxsame}.
\end{proof}

The above corollary allows us to only consider the case when $u$ is the longest element of some $W_I$.

\begin{proposition}
\label{prop:main}
In the above decomposition, if $v$ is a locally longest element or is in a local chain, then $P_u(q)=R_u(q)$ implies that $P_w(q) = R_w(q)$.
\end{proposition}
\begin{proof}

From the previous Corollary, it is enough to show that $\A_{u_{I \cap J}v}$ is uniform with respect to $\A_{u_{I \cap J}}$ and that $R_{u_{I \cap J}v}(q) = P_{u_{I \cap J}v}(q)$.

If $v$ is the longest element of $^JW$, then $u_{I \cap J}v$ is the longest element of $W_I$. The claim follows from Lemma~\ref{lem:maxsame}.

When $W_{I}^{I \cap J}$ is a chain, let $v'$ denote the longest element of $^{I \cap J}W_{I}$. Then $w' \defeq u_{I \cap J}v'$ is the longest element of $W_I$. From Lemma~\ref{lem:maxsame}, we know that $R_{u_{I \cap J}}(q) = P_{u_{I \cap J}}(q)$ and $R_{u_{I \cap J}v'}(q) = P_{u_{I \cap J}v'}(q)$. For each chamber $u$ of $\A_{u_{I \cap J}}$, the poset $Q(\A_{w'},\A_{u_{I \cap J}},u)$ is a chain of length $\ell(v')$. In particular, every hyperplane of $\A_{w'} \setminus \A_{u_{I \cap J}}$ intersects the interior of the chamber $u$.

When we go from $\A_{w'=u_{I \cap J}v'}$ to $\A_{u_{I \cap J}v}$, we are removing some hyperplanes of $\A_{w'} \setminus \A_{u_{I \cap J}}$. For each chamber $u$ of $\A_{u_{I \cap J}}$, the poset $Q(\A_{u_{I \cap J}},\A_{u_{I \cap J}},u)$ is a chain of length $\ell(v')$ minus the number of hyperplanes removed. Hence $\A_{u_{I \cap J}v}$ is uniform with respect to $\A_{u_{I \cap J}}$. Moreover, we have $R_{u_{I \cap J}v}(q) = R_{u_{I \cap J}(q)}(1+\cdots+q^{\ell(v)})$. The claim now follows from Lemma~\ref{lem:maxsame}.
\end{proof} 

Lastly we analyze two special examples each coming from $F_4$ and $B_n$ which will be needed in the next section.

\begin{example}
\label{ex:F4special}
Let $w \in F_4$ be $w=uv$ where $u$ is the longest element of $W_{\{1,2,3\}}$ and $v = s_4s_3s_2s_1$. Then $P_w(q) = P_{u}(q) (1+q+q^2+q^3)$. The hyperplane arrangement $\A_w$ is taking the hyperplanes $x_1=0,x_2=0,x_3=0,x_2-x_1=0,x_3-x_2=0,x_3-x_1=0,x_1+x_2=0,x_1+x_3=0,x_2+x_3=0$ coming from $\A_{u}$ and additionally taking the hyperplanes $-x_1-x_2-x_3+x_4=0, -x_1-x_2+x_3-x_4=0, -x_1+x_2-x_3-x_4=0, x_1-x_2-x_3-x_4=0$. 

Pick any chamber $c$ of $\A_u$ and pick an arbitrary interior point $x$ inside. Chamber $c$ determines a total order on $x_1,\ldots,x_{n-1},-x_1,\ldots,-x_{n-1}$ that does not depend on the choice of $x$. Consider the line $l_x$ obtained from $x$ by changing the $x_n$ value from $-\infty$ to $+\infty$. This line is still contained in chamber $c$. Imagine moving through the line $l_x$ by changing the $x_n$ value from $-\infty$ to $+\infty$. The order we cross the hyperplanes of $\A_w \setminus \A_{u}$ is determined by the total order on $x_1,\ldots,x_{n-1},-x_1,\ldots,-x_{n-1}$. 

Hence $\A_w$ is uniform with respect to $\A_u$. Moreover, the poset $Q_w$ is obtained from $Q_u$ by doing a product with a chain of length $4$. We get $R_w(q) = R_{u}(q) (1+ \cdots + q^4)$. Since $R_{u}(q) = P_{u}(q)$ from Lemma~\ref{lem:maxsame} and $P_v^{^JW}(q) = (1+ \cdots + q^{\ell(v)})$, we obtain the desired result.

%From the help of computers, the distance enumerating polynomial $R_w(q)$ of $\A_w(q)$ turns out to be same as $P_w(q)$. By symmetry, when $u_0$ is the longest element of $W_{2,3,4}$ and $v=s_1s_2s_3s_4$, we have $R_w(q) = P_w(q)$ as well.
\end{example}

\begin{lemma}
\label{lem:Bnspecial}
Let $W$ be a type $B_n$ Weyl group and $S$ be the set of simple roots, corresponding to the Dynkin diagram of Figure~\ref{dia:Bn}. Set $J = S \setminus \{s_0\}$. Pick $w = u v$ where $u$ is the longest element of $W_J$ and $v = s_0 s_1 \cdots s_{n-1}$. Then $\A_w$ is unifrom with respect to $\A_u$ and $P_w(q) = R_w(q)$.
\end{lemma}
\begin{proof}
The hyperplane arrangement $\A_u$ consists of hyperplanes $x_i = 0$ for $1 \leq i \leq n-1$, hyperplanes $x_i - x_j = 0$ for $1 \leq i < j \leq n-1$ and hyperplanes $x_i + x_j = 0$ for $1 \leq i < j \leq n-1$. The hyperplane arrangement $\A_w$ is obtained from $\A_u$ by adding in the hyperplane $x_n=0$ and hyperplanes $x_n + x_i = 0$ for $1 \leq i \leq n-1$. 

Pick any chamber $c$ of $\A_u$ and pick an arbitrary interior point $x$ inside. Chamber $c$ determines a total order on $x_1,\ldots,x_{n-1}$ and $0$ that does not depend on the choice of $x$. Consider the line $l_x$ obtained from $x$ by changing the $x_n$ value from $-\infty$ to $+\infty$. This line is still contained in chamber $c$. Imagine moving through the line $l_x$ by changing the $x_n$ value from $-\infty$ to $+\infty$. The order we cross the hyperplanes of $\A_w \setminus \A_u$ is determined by the total order on $x_1,\ldots,x_{n-1}$ and $0$.

Hence $\A_w$ is uniform with respect to $\A_u$. Moreover, the poset $Q_w$ is obtained from $Q_u$ by doing a product with a chain of length $n$. We get $R_w(q) = R_u(q) (1+ \cdots + q^{\ell(v)})$. Since $R_u(q) = P_u(q)$ from Lemma~\ref{lem:maxsame} and $P_v^{^JW}(q) = (1+ \cdots + q^{\ell(v)})$, we obtain the desired result.

%Let $\alpha_1,\ldots,\alpha_{n-1}$ denote the coordinates $\{1,\ldots,n-1\}$ in order where the coordinate entry of $x$ is from biggest to smallest. Then this permutation $\alpha$ is independent of the choice of $x$ inside $c$ (Since each chamber of $\A_u$ is basically $\epsilon_{\beta_1}x_{\beta_1} > \cdots > \epsilon_{\beta_{n-1}}x_{\beta_{n-1}} > 0$ for some permutation $\beta$ of $\{1,\ldots,n-1\}$ and $\epsilon$'s each being a choice of $+1$ or $-1$).

%Consider the line $l_x$ obtained from $x$ by changing the $x_n$ value from $-\infty$ to $+\infty$. This line is still contained in chamber $c$. Imagine moving through the line $l_x$ by changing the $x_n$ value from $-\infty$ to $+\infty$. We first will cross $x_n + x_{\alpha_1}=0$, then $x_n + x_{\alpha_2}$ and so on and will cross $x_n + x_{\alpha_{n-1}}=0$ lastly. Hence $Q_w$ is obtained from $Q_u$ by doing a product with a chain of length $n$. We get $R_w(q) = R_u(q) (1+ \cdots + q^{\ell(v)})$. Since $R_u(q) = P_u(q)$ from Lemma~\ref{lem:maxsame} and $P_v^{^JW}(q) = (1+ \cdots + q^{\ell(v)})$, we get the desired result.

\end{proof}

\section{Conclusion and further remarks}

In this section, we finally prove that $P_w$ and $R_w$ are the same when $w$ is palindromic.

\begin{theorem}
Let $W$ be a Weyl group. Let $w$ be a rationally smooth element. Then $R_w(q) = P_w(q)$. 
\end{theorem}
\begin{proof}

We use induction on the rank of $W$. The claim is obvious for rank $1$. Decompose $w$ or $w^{-1}$ via BP-decomposition. By applying Theorem~\ref{thm:quotpalin}, we see that either $v$ is a locally-longest element or is in a local chain or is in special cases of $F_4$ or $B_n$. In the first two cases, that is when $v$ is a locally-longest element or is in a local chain, then Proposition~\ref{prop:main} allows us to replace $w$ with rationally smooth $u$ of strictly smaller rank. If we are in the special cases, using Example~\ref{ex:F4special} and Lemma~\ref{lem:Bnspecial} combined with Corollary~\ref{cor:main} allows us the same replacement.
\end{proof}

Since $R_w(q)$ is always palindromic by definition, we get the following result as a corollary:

\begin{corollary}
Let $W$ be a Weyl group. Then $w$ is rationally smooth if and only if $P_w(q) = R_w(q)$.
\end{corollary}

We would like to mention that \cite{SLOF} has an explanation of the factors of $R_w(q)$ using the structure of the hyperplane arrangement.

The proof of the main theorem is based on a recurrence relation, using Theorem~\ref{thm:quotpalin}. Back in \cite{OY}, we conjectured that Corollary~\ref{cor:conj} would be true for any Coxeter groups. Although Theorem~\ref{thm:quotpalin} tells us this is true for Weyl groups (also answered in \cite{SLOF}), it turns out that there are palindromic lower intervals in a leaf quotient of $H_4$ Coxeter groups that does not satisfy the property.

\begin{figure}[h!]
\centering\begin{tikzpicture}
\draw[line width=0.7pt]
	(0,0)--(6,0);
\draw[fill=white, line width=1pt]
	(0,0) circle[radius=1.8mm]
	(2,0) circle[radius=1.8mm]
	(4,0) circle[radius=1.8mm]
	(6,0) circle[radius=1.8mm];
\node at (0,-.5) {$s_4$};
\node at (2,-.5) {$s_3$};
\node at (4,-.5) {$s_2$};
\node at (6,-.5) {$s_1$};
\node at (1,.5) {$5$};
\end{tikzpicture}
\caption{Type $H_4$ Coxeter system}\label{dia:H4}
\end{figure}
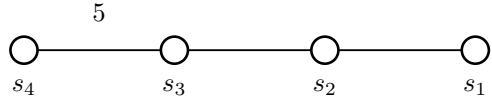

\begin{example}[$H_4$]
\label{ex:H4}
The following are the elements of $^{S \setminus \{s\}}W$ when $W$ is a Coxeter group of type $H_4$, $S$ the simple roots of $W$ and $s$ a leaf of the Coxeter diagram as in Figure~\ref{dia:H4}. We only list the nontrivial and non-chain palindromic elements. The result was obtained by using the Coxeter package developed by Stembridge as before. For this example, we use $s_i$ and $i$ interchangeably, so a list of numbers corresponds to a word of simple roots.

\begin{enumerate}
\item $W=H_4, s=1$
\begin{itemize}
\item All palindromic elements here are the longest element or a chain element.
\end{itemize}

\item $W=H_4, s=4$
\begin{itemize}
\item $[4,3,4,2,3,4,1,2,3,4,2,3,4,2,3,4]$
\item $[4,3,4,2,3,4,2,3,4,2,3,4]$
\item $[4,3,4,2,3,4,3]$
\end{itemize}

The element $[4,3,4,2,3,4,1,2,3,4,2,3,4,2,3,4]$ uses all simple roots, so is impossible to be a longest word of an embedded quotient.

\end{enumerate}

\end{example}

One nice property that $R_w(q)$ has is that it is always palindromic regardless of the rational smoothness of $w$. And this is a property that intersection homology Poncar\'e polynomial $IP_w(q)$ also has. So it would be interesting to compare these two polynomials \cite{KOO}.

\bibliographystyle{plain}    % (uses file "plain.bst")
\bibliography{bib}

\end{document}